\documentclass[11pt]{amsart}
\usepackage{amsmath}
\usepackage{verbatim}
\usepackage{url}

\usepackage[all, cmtip,arrow]{xy}

\usepackage{pb-diagram,pb-xy}

\usepackage{color}

\newcommand{\hs}{\kern 0.75pt}
\newcommand{\hm}{\kern -0.75pt}

\usepackage{amssymb}

\numberwithin{equation}{section}

\newtheorem{thm}[equation]{Theorem}
\newtheorem{prop}[equation]{Proposition}

\newtheorem{lemma}[equation]{Lemma}
\newtheorem{cor}[equation]{Corollary}
\newtheorem{con}[equation]{Conjecture}

\theoremstyle{definition}
\newtheorem{rem}[equation]{Remark}

\newtheorem{dfn}[equation]{Definition}

\newcommand{\codim}{\operatorname{codim}}

\newcommand{\SB}{\mathop{\mathrm{SB}}}

\newcommand{\Br}{\mathop{\mathrm{Br}}}

\newcommand{\Pic}{\mathop{\mathrm{Pic}}}
\newcommand{\ind}{\mathop{\mathrm{ind}}}

\newcommand{\CH}{\mathop{\mathrm{CH}}\nolimits}

\newcommand{\PGO}{\operatorname{\mathrm{PGO}}}

\newcommand{\PGL}{\operatorname{\mathrm{PGL}}}

\newcommand{\GL}{\operatorname{\mathrm{GL}}}

\newcommand{\Ch}{\mathop{\mathrm{Ch}}\nolimits}

\newcommand{\res}{\mathop{\mathrm{res}}\nolimits}

\newcommand{\Z}{\mathbb{Z}}
\newcommand{\F}{\mathbb{F}}

\newcommand{\C}{\mathbb{C}}

\newcommand{\Aut}{\operatorname{Aut}}

\newcommand{\pt}{\mathbf{pt}}

\newcommand{\Prod}{\operatornamewithlimits{\textstyle\prod}}
\newcommand{\Sum}{\operatornamewithlimits{\textstyle\sum}}
\newcommand{\Oplus}{\operatornamewithlimits{\textstyle\bigoplus}}

\newcommand{\Label}{\label}

\newcommand{\Ker}{\operatorname{Ker}}

\renewcommand{\phi}{\varphi}

\title
{The $J$-invariant over splitting fields of Tits algebras}
\keywords
{Chow motives, $J$-invariant, Algebraic groups, Tits algebras.}

\author{Maksim Zhykhovich}

\address{Zhykhovich: Mathematisches Institut, Ludwig-Maximilians-Universit\"at M\"unchen,
Theresienstr. 39, D-80333  M\"unchen, Germany}

\email{zhykhovich@math.lmu.de}

\begin{document}

\begin{abstract}
We describe the $J$-invariant of a semi-simple algebraic group $G$ over a generic splitting field of a Tits algebra of $G$ in terms of the $J$-invariant over the base field. As a consequence we prove a 10 years old conjecture of  Qu\'eguiner-Mathieu, Semenov and Zainoulline  on the \mbox{$J$-invariant} of groups of type $\mathrm{D}_n$. In the case of type $\mathrm{D}_n$ we also provide explicit formulas for the first component and in some cases for the second component of the $J$-invariant.


\end{abstract}

\maketitle

\section{Introduction}

Chow motives were introduced by Grothendieck, and since then they became a fundamental tool for investigating the structure of algebraic varieties. The study of Chow motives and motivic decompositions has several outstanding applications to other topics. For example, Voevodsky's proof of the Milnor conjecture relied on Rost's computation of the motivic decomposition of a Pfister quadric. In \cite{Ka10} Karpenko established the relation between the motivic decomposition and the canonical dimension of a projective homogeneous variety, which allowed to compute the canonical dimension in many cases.

In \cite{PSZ} Petrov, Semenov, and Zainoulline  investigated the structure of the motives of generically split projective homogeneous varieties and introduced a new invariant of an algebraic group $G$,
called the {\it $J$-invariant}. In the case of quadratic forms the $J$-invariant was introduced previously by Vishik in \cite{Vi05}. 
For a fixed prime number $p$ the $J$-invariant of $G$ modulo $p$ is a discrete invariant consisting of several non-negative integer components  $(j_1, ..., j_r)$ with degrees $1\leq d_1\leq \dots \leq d_r $. The integers $r$ and $d_1, \dots d_r$ depend only on the type of $G$ and are known for all types (see table \cite[\S 4.13]{PSZ}). The $J$-invariant encodes the motivic decomposition of the variety $X$ of Borel subgroups in $G$. More precisely, it turns out that the Chow motive of $X$ with coefficients in $\F_p$ decomposes into a direct sum of Tate twists of an indecomposable motive $\mathcal{R}_p(G)$ and the Poincar\'e polynomial of $\mathcal{R}_p(G)$  over a splitting field of $G$ equals
\begin{equation}
\prod_{i=1}^r\frac{t^{d_ip^{j_i}}-1}{t^{d_i}-1}\in\Z[t].
\end{equation}

The $J$-invariant proved to be an important tool for solving several long-standing problems. For example, it plays an important role in the progress on the Kaplansky problem about possible values of the $u$-invariant, see \cite{Vi07} . Another example is the proof of a conjecture of Serre about groups of type $\mathrm{E}_8$ and its finite subgroups, where the $J$-invariant plays a crucial role (see \cite{GS10} and \cite{Sem}). More recently, Petrov and Semenov generalized the $J$-invariant for groups of inner type to arbitrary oriented cohomology theories in the sense of Levine-Morel \cite{LM} satisfying some axioms (see \cite{PS21}). 


Let $(A,\sigma)$ be a central simple algebra of even degree $2n$, endowed with an involution of orthogonal type and trivial discriminant. Let $G= \PGO^+(A,\sigma)$ be the connected component of the automorphism group of $(A, \sigma)$. The group $G$ is adjoint of type $\mathrm{D}_n$. Denote by $J(G) = (j_1, ..., j_r)$ the $J$-invariant of $G$ modulo $p=2$. It is known that the first component $j_1$ is zero  if the algebra $A$ is split. In particular, $j_1$ becomes zero over the function field $F_A$ of the Severi-Brauer variety of $A$, which is a generic splitting field of $A$.

 In \cite{QSZ} Qu\'eguiner-Mathieu, Semenov and Zainoulline
stated a conjecture that the remaining components do not change after generic splitting of $A$.

\begin{con}[{\cite[Remark 7.3]{QSZ}}]
\Label{conj}
If $J(G)=(j_1, ..., j_r)$, then $j_i=(j_i)_{F_A}$ for $i=2, ... ,r$.
\end{con}
Note that, in the settings of Conjecture \ref{conj}  the central simple algebra $A$ is a Tits algebra of the algebraic group $\PGO^+(A,\sigma)$.
In the present paper we prove Conjecture \ref{conj} and, moreover, generalize it to the case of an arbitrary semi-simple algebraic group $G$ of inner type.  Let $A$ be a Tits algebra of $G$. The main result of the paper (Theorem \ref{main}) describes the connection between the  $J$-invariant of $G$ over a generic splitting field of $A$ and the $J$-invariant over a base field. In particular, we prove that all components of the $J$-invariant of $G$ of degree greater than $1$ do not change after extending to a generic splitting field $F_A$ of $A$. 
Moreover, the main theorem provides some control on how the components of degree $1$ can change over the field $F_A$. In the case $G=\PGO^+(A,\sigma)$ we improve this control in Proposition \ref{j_1=jGA}, which together with the main theorem allows to prove Conjecture \ref{conj} (see Corollary \ref{corconj}).


The main result of the paper allows to split a Tits algebra of an algebraic group without losing much information on the $J$-invariant of the group (only components of degree one may be affected). This may be a useful tool to compute the $J$-invariant, since the algebraic groups with trivial Tits algebras are considered as ``less complex objects'' comparing to those groups with non-trivial Tits algebras.
For example, in the settings of Conjecture \ref{conj} the group $\PGO^+(A,\sigma)$ over the field $F_A$ becomes isomorphic to $\PGO^+(q_{\sigma})$, where $q_{\sigma}$ is the respective quadratic form adjoint to the split algebra with involution $(A, \sigma)_{F_A}$. Hence, Conjecture \ref{conj} allows to reduce the computation of the $J$-invariant (except the first component) of algebras with orthogonal involution to the case of quadratic forms. Note that recently the similar approach was used to investigate the motivic equivalence of algebras with involutions, see \cite{CQZ}.



Section 5 of the paper is devoted to the computation of the first two components $j_1$ and $j_2$ of the $J$-invariant of the group $\PGO^+(A,\sigma)$ of type $\mathrm{D}_n$, where $(A, \sigma)$ is a central simple algebra with orthogonal involution. This question was already investigated in \cite{QSZ}. Namely, in \cite[Corollary 5.2]{QSZ} the upper bounds for $j_1$ and $j_2$ were provided in terms of $2$-adic valuations $i_A$, $i_+$, $i_-$ of indices of algebras   $A$, $C_+$ and $C_-$ respectively, where $C_+$ and $C_-$ are the components of the Clifford algebra of $(A, \sigma)$.
We improve the upper bound for $j_1$ and then show that it is in fact the exact value of $j_1$ (see Theorem \ref{j1}). More precisely, we obtain the following formula
\begin{equation} 
j_1= \min\{k_1, i_A, \max\{i_+, i_-\}\} \, ,
\end{equation}
\noindent where $k_1$ denotes the $2$-adic valuation of $n$.

Note that the formula for $j_1$ and Conjecture \ref{conj} (Corollary \ref{corconj}) allow to {\it completely} reduce the computation of the $J$-invariant of the group $\PGO^+(A,\sigma)$ to the case of quadratic forms. 
Moreover, in Proposition \ref{j2} we also provide an explicit formula for the second component $j_2$ in some cases. 

Note that recently Henke in his PhD thesis  applied the main result of this paper to investigate  motivic decompositions of projective homogeneous varieties  for groups of type  $\mathrm{E}_7$.

The proofs in this paper rely on the computations of rational cycles, the properties of generically split varieties, the theory of upper motives and the index reduction formula.

\section{Preliminaries. Notation.}
\subsection{Chow motives}

Let $F$ be a field. In the present paper we work in the category of the Grothendieck-Chow motives over $F$ with coefficients in $\F_p$ for a fixed prime number $p$ (see \cite{EKM}).

For a smooth projective variety $X$ over $F$ we denote by $M(X)$ the motive of $X$ in this category. We consider the Chow ring $\CH(X)$ of $X$ modulo rational equivalence and we write $\Ch(X)$ for the Chow ring with coefficients in $\F_p$.

For a motive $M$ over $F$ and a field extension $E/F$ we denote by $M_E$ the extension of scalars. A motive $M$ is called {\it split} (resp. {\it geometrically split}), if it is isomorphic to a finite direct sum of Tate motives (resp. if $M_E$ is split over some field extension $E/F$).

Let $G$ be a semisimple algebraic group of inner type. Let $X$ be a projective homogeneous variety under the action of $G$. Note that the motive of $X$ splits over any field extension, over which the group $G$ splits (in the sense of algebraic groups). By $\overline{X}$ we denote the variety $X_E$ over a splitting field $E$ of the group $G$. The Chow ring $\CH(\overline{X})$ does not depend on the choice of $E$ and, therefore, we do not specify the splitting fields in the formulas below. By $\overline{\CH}(X)$ we denote the image of the restriction homomorphism $\CH(X) \rightarrow \CH(\overline{X})$. We say that a cycle from $\CH(\overline{X})$ is $F$-rational if it belongs to $\overline{\CH}(X)$.

The {\it Poincar\'e polynomial} $P(X,t)$ of $X$ is defined as $\Sum_{i\geq 0} \dim \Ch^i(\overline{X}) t^i$. Similarly, for direct motivic summand $M$ of $X$ we define the Poincar\'e polynomial $P(M,t)$ by replacing $\overline{X}$ by $\overline{M}$ in the formula, where $\overline{M}$ denotes the motive $M$ over a splitting field of $G$.


Recall that the {\it Krull-Schmidt principle} holds for any motivic direct summand $M$ of a projective homogeneous variety $X$. Namely,  $M$ decomposes in a unique way in a finite direct sum of indecomposable motives, see \cite{ChM06}. The {\it upper motive } $U(X)$ of $X$ is defined as an indecomposable summand  of $M(X)$ with the property that the Chow group $\Ch^0(U(X))$ is non-zero. It follows by the Krull-Schmidt principle that the isomorphism class of $U(X)$ is uniquely determined by $X$.

Given two projective homogeneous varieties $X_1$ and $X_2$ (under possibly different algebraic groups) over $F$, the upper motives of the varieties $X_1$ and $X_2$ satisfy the following isomorphism criterion.

\begin{prop}[{\cite[Corollary 2.15]{Ka13}}]
\label{prop-upper}
The upper motives $U(X_1)$ and $U(X_2)$ are isomorphic if and only of the varieties $X_1$ and $X_2$ possesses $0$-cycles of degree $1$ modulo $p$ respectively over $F(X_2)$ and $F(X_1)$.
\end{prop}

\subsection{Tits algebras and the Picard group} 

Let $G_0$ be a split semisimple algebraic group of inner type of rank $n$ over $F$. We fix a split maximum torus $T$ in $G_0$ and a Borel subgroup $B$ of $G_0$ containing $T$. Let $\Pi =\{\alpha_1, ..., \alpha_n\}$ be a set of simple roots with respect to $B$ and let $\{\omega_1,...,\omega_n\}$ be the respective set of fundamental weights.  Enumaration of roots and weight follows Bourbaki.

Denote by $\Lambda_{\omega}$ the respective weight lattice and by $\Lambda_{\omega}^+$ the cone of dominant weights. There is a natural one-to-one correspondence between the isomorphism classes of the irreducible finite dimensional representations of $G_0$ and the elements of $\Lambda_{\omega}^+$. This correspondence associates with an irreducible representation of $G_0$ its highest weight.

Let now $G$ be an arbitrary (not necessarily split) semisimple algebraic group over $F$ of inner type which is a twisted form of $G_0$. With each element $\omega \in \Lambda_{\omega}^+$ one can associate a unique central simple algebra $A_{\omega}$ such that there exists a group homomorphism $\rho: G \rightarrow \GL_1(A)$ having the property 
that the representaion $\rho \otimes F_{\text{sep}}$ of the split group $G\otimes F_{\text{sep}}$ is  the representation with the highest weight $\omega$. The algebra $A_{\omega}$ is called the {\it Tits algebra} of $G$ corresponding to $\omega$. In particular, to any fundamental weight $\omega_i$ corresponds a Tits algebra $A_{\omega_i}$.

Any projective homogeneous $G$-variety $X$ is the variety of parabolic subgroups in $G$ of some fixed type, where the type corresponds to a subset $\Theta$ of the set of simple roots $\Pi$.
The Picard group $\Pic(\overline{X})$ can be identified with a free $\Z$-module generated by $\omega_i$, $i \in \Pi \backslash \Theta$. Consider the group homomorphism $\alpha_X: \Pic(\overline{X}) \rightarrow \Br(F)$ sending $\omega_i$ to the Brauer-class of the Tits algebra $A_{\omega_i}$ corresponding to the fundamental representation with the highest weight $\omega_i$.

By \cite[\S 2]{MT95} the following sequence of groups is exact
\begin{equation}
\label{Br}
    0 \rightarrow \Pic(X) \xrightarrow{\res} \Pic(\overline{X}) \xrightarrow{\alpha_X} \Br(F) \, ,
\end{equation}
where $\res$ is the scalar extension to a splitting field of $G$.

\noindent This sequence allows to express the group $\Pic(X)$ in terms of the Tits algebras of $G$.
\section{$J$-invariant}
\label{J-invariant}
The $J$-invariant of a semisimple algebraic group was introduced in \cite{PSZ} by Petrov, Semenov and Zainoulline.
In this section we briefly recall the definition and the main properties of the $J$-invariant following \cite{PSZ}.

Let $G_0$ be  a split semisimple algebraic group over a field $F$ and $B$ a Borel subgroup of $G_0$. 
An explicit presentation of $\Ch^*(G_0)$ in terms of generators and relations is known for all groups and all primes $p$. Namely, by \cite[Theorem 3]{Kac85}
\begin{equation}
\label{ChG}
\Ch^*(G_0) \simeq \F_p[e_1, ...,e_r]/(e^{p^{k_1}}_1, \dots , e^{p^{k_r}}_r )
\end{equation}

\noindent for some non-negative integers $r$, $k_i$ and some homogeneous generators $e_1, ..., e_r$ with degrees  $1\leq d_1\leq \dots \leq d_r $ coprime to $p$. A complete list of numbers $r$, $k_i$ and $d_i$ is provided in \cite[page 21]{PSZ} for any split group $G_0$.

We introduce an order on the set of additive generators of $\Ch^*(\overline{G})$, i.e., on the monomials $e_1^{m_1}\dots e_r^{m_r}$. To simplify the notation, we denote the monomial 
$e_1^{m_1}\dots e_r^{m_r}$ by $e^M$, where $M$ is an $r$-tuple of integers $(m_1, \dots, m_r)$. The codimension (in the Chow ring) of $e^M$ is denoted by $|M|$. Note that $|M|= \sum_{i=1}^r d_im_i$.

Given two $r$-tuples $M= (m_1, \dots, m_r)$ and $N= (n_1, \dots, n_r)$ we say $e^M \leq e^N$ (or equivalently $M \leq N$) if either $|M|<|N|$ or $|M|=|N|$ and $m_i\leq n_i$ for the greatest $i$ such that $m_i \neq n_i$. This gives a well-ordering on the set of all monomials ($r$-tuples).

Let now $G = {}_{\xi}G_0$ be an inner twisted form of $G_0$ given by a cocycle $\xi \in Z^1(F, G_0)$ and let $X={}_{\xi}(G/B)$ be the variety of Borel subgroups in $G$. Since $X$ and $G/B$ are isomorphic over any splitting field of $G$, we identify the Chow groups $\Ch(\overline{X})$ and $\Ch(G_0/B)$.


We consider the following composite map
\begin{equation}
\label{J-sequence}
\Ch^*(X)\xrightarrow{\res} \Ch^*(G_0/B) \xrightarrow{\pi} \Ch^*(G_0),
\end{equation}
where $\pi$ is the surjective pullback of the canonical projection $G_0\to G_0/B$ and $\res$ is the scalar extension to a splitting field of $G$.

\begin{dfn}[{\cite[Definition 4.6]{PSZ}}]
For each $i$, $1 \leq i \leq r$, set $j_i$ to be the smallest non-negative interger such that the image of the composite map $\pi \circ \res$ contains an element $a$ with the greatest monomial $e_i^{p^{j_i}}$ with respect to the order on $\Ch^*(G_0)$ as above, i.e., of the form
$$a= e_i^{p^{j_i}} + \Sum_{e^M < e_i^{p^{j_i}}} c_M e^M \, ,\quad c_M \in \F_p \, .$$
\noindent The $r$-tuple of integers $(j_1, \dots, j_r)$ is called the {\it $J$-invariant of $G$ modulo $p$} and is denoted by $J_p(G)$.
\end{dfn}

\begin{rem}
Note that the $J$-invariant of $G$ up to a permutation of some components may depend on the choice of the cocycle $\xi$ (see \cite[\S 3]{QSZ}). Below by considering the $J$-invariant of a group $G$ we also fix a cocycle $\xi$.

\end{rem}

\begin{rem}
\Label{generic}
According to \cite[Proposition 5.1]{GiZ12}  and \cite[Theorem 6.4]{KM06} the sequence (\ref{J-sequence}) of graded rings is exact in the middle term if and only if the cocycle $\xi$ is generic.
\end{rem}

In \cite{PSZ} the following motivic interpretation of the $J$-invariant was provided.
\begin{prop}[{\cite[Theorem 5.13]{PSZ}}] 
\label{gen-split}
Let $X$ be the variety of Borel subgroups in $G$. Then the Chow motive of $X$ with coefficients in $\F_p$ decomposes into a direct sum 
\begin{equation}
\label{M(X)}
M(X)\simeq \Oplus_{i\in I} \mathcal{R}_p(G)(i)
\end{equation}
\noindent of twisted copies of an indecomposable motive $\mathcal{R}_p(G)$ for some finite multiset $I$ of non-negative integers. Moreover, the Poincar\'e polynomial of $\mathcal{R}_p(G)$ is given by
\begin{equation}
\label{formula}
P(\mathcal{R}_p(G), t) = \Prod_{i=1}^{r} \frac{1-t^{d_ip^{j_i}}}{1-t^{d_i}}\, .
\end{equation}
\end{prop}

\begin{rem}
Note that the $J$-invariant allows to compute not only the Poincar\'e polynomial of $\mathcal{R}_p(G)$  by formula \ref{formula} but also all twisting numbers in the motivic decomposition (\ref{M(X)}) of $X$. Namely, we have
$$ \frac{P(X,t)}{P(\mathcal{R}_p(G), t)} = \Sum_{i \geq 0} a_it^i \, ,$$
\noindent where $a_i$ is the number of the copies of $\mathcal{R}_p(G)$ with the twisting number $i$ in the motivic decomposition (\ref{M(X)}). Note that the Poincar\'e polynomial $P(X,t)$ can be explicitly computed by Solomon formula (see \cite[9.4 A]{Car}). 
\end{rem}
 The motivic decomposition from Proposition \ref{gen-split} holds for more general class of varieties. Namely, a projective homogeneous $G$-variety $X$ is called {\it generically split} if the group $G$ splits over the generic point of $X$. In particular, the variety of Borel subgroups in $G$ is generically split. Another examples are Pfister quadrics and Severi-Brauer varieties. By \cite[Theorem 5.17]{PSZ} the Chow motive of any generically split $G$-variety with coefficients in $\F_p$ decomposes into a direct sum of twisted copies of the motive $\mathcal{R}_p(G)$.
\begin{prop}[{\cite[Theorem 5.5]{PS1}}] 
\label{Poincare-rat}
Let $X$ be a generically split $G$-variety and recall that $\overline{\Ch}(X)$ denotes the subring of $F$-rational cycles in $\Ch(X)$. Then 
$$ \frac{P(X,t)}{P(\mathcal{R}_p(G), t)}  = P(\overline{\Ch}(X), t) \, .$$
\noindent 
In particular, the number of copies of the motive $\mathcal{R}_p(G)(i)$ in the complete motivic decomposition of $X$ is equal to $\dim_{\, \F_p} \overline{\Ch}^i(X)$.

\end{prop}

\begin{rem}
The above proposition shows that the subring $\overline{\Ch}(X) \subset \Ch(\overline{X})$ of $F$-rational cycles encodes the complete motivic decomposition of a generically split variety $X$ over $F$. In contrast, in general in order to find the complete motivic decomposition of a projective homogeneous $G$-variety $X$ one usually needs to describe $F$-ration projectors in the Chow group of the product $\overline {X \times X}$.  
\end{rem}

We finish this section with several observations, which will be usefull later in this paper.
\begin{lemma}
\label{ax1}
Let $\mathfrak{X}, \mathcal{Y}$ be two projective homogeneous varieties over a field $F$. Assume that $\mathcal{Y}$  possesses a zero-cycle of degree $1$ over the function field $F(\mathfrak{X})$. Then the cycle $a\in \Ch(\overline{\mathfrak{X}})$ is $F$-rational if and only if the cycle $a\times 1\in \Ch(\overline{\mathfrak{X} \times \mathcal{Y}} )$ is $F$-rational.
\end{lemma}
\begin{proof}
The direct implication is clear. To show the inverse implication we assume that $a\times 1\in \Ch(\overline{\mathfrak{X} \times \mathcal{Y}} )$ is $F$-rational.

Let $\alpha \in \Ch(\mathfrak{X} \times \mathcal{Y})$ be a cycle, such that  $\overline{\alpha} = a\times 1$.  Let $x \in \Ch_0(\mathcal{Y}_{F(\mathfrak{X})})$ be a zero-cycle of degree $1$. Let $\beta \in \CH(X \times \mathcal{Y})$ be a preimage of $x$ under the flat pull-back
$$\Ch(\mathfrak{X} \times \mathcal{Y}) \, \, -\!
\!\! \twoheadrightarrow \Ch(\mathcal{Y}_{F(\mathfrak{X})})$$
\noindent along the morphism induced by the generic point of $\mathfrak{X}$. Since $\overline{\beta}= 1 \times [\pt] + \Sum_{i\in I}a_i\times b_i$, where $\dim b_i>0$ for all $i \in I$, we have 
 $$\overline{(pr_\mathfrak{X})_*(\alpha \beta)} =  (pr_\mathfrak{X})_*(\overline{\alpha\beta}) = (pr_\mathfrak{X})_* (a \times [\pt] +  \Sum_{i\in I}(aa_i\times b_i)) =a \, ,$$
 where $pr_\mathfrak{X}$ is the projection $\mathfrak{X}\times \mathcal{Y} \rightarrow \mathfrak{X}$ on the first factor.
 
 It follows from the above equality that the cycle $a$ is $F$-rational.

\end{proof}

\begin{dfn}
\label{jGA}
Let $G$ be a semisimple algebraic group of inner type, $A$ a Tits algebra of $G$ and $X$ a projective homogeneous variety such that 
$U(X) \simeq R_p(G)$. 
Let $w \in \Ch^1(\overline{X})$ be a cycle such that for some lifting $\tilde{w} \in \CH^1(\overline{X})$ of $w$ holds $\alpha_X(\tilde{w})=l [A] \in \Br(F)$, where $l$ is an integer coprime to $p$. We define $j_{G,A}$ to be the smallest integer $j\geq 0$ such that the cycle $w^{p^j} \in \Ch(\overline{X})$ is $F$-rational.

\end{dfn}

\begin{rem}
Note that the definition of $j_{G,A}$ does not depend on the choice of the variety $X$ and the cycle $w$. Indeed, if $Y$ and $v$ is another choice of a variety and a cycle satisfying the above conditions, then using the exact sequence (\ref{Br}) for the variety $X\times Y$
we obtain an $F$-rational cycle $w\times 1 -\lambda (1\times v) \in \Ch(\overline{X\times Y})$ for some $\lambda \in \F_p^*$. Then for every $j \geq 0$ we have 
$$(w\times 1 -\lambda (1\times v))^{p^{j}} = w^{p^{j}}\times 1 -\lambda (1\times {v}^{p^{j}}) \in \overline{\Ch}(X\times Y) \, .$$
Hence, we get the  following equivalences 
$$w^{p^{j}} \in \overline{\Ch}(X) \iff w^{p^{j}}\times 1 \in \overline{\Ch}(X \times Y) \iff 1\times {v}^{p^{j}} \in \overline{\Ch}(X \times Y) \iff  {v}^{p^{j}} \in \overline{\Ch}( Y) \, , $$
where the first and the last equivalences hold by Lemma \ref{ax1}.

\end{rem}

\section{Main result}
Let $F$ be a field and $p$ a prime number.
Let $G$ be a semi-simple algebraic group over a field $F$ of inner type.
Let $J(G)=(j_1, ..., j_r)$ be the $J$-invariant of $G$ modulo $p$ and let $d_1\leq \dots \leq d_r $ be the respective degrees of the components. Denote by $J^1(G) = (j_1,...,j_l)$, $l \leq r$, the family of all components of degree $1$, that is $d_l=1$ and $d_{l+1}>1$.

Let $A$ be a Tits algebra of $G$. Denote by $F_A$ the function field of the Severi-Brauer variety $\SB(A)$. The main theorem of this paper describes the $J$-invariant of $G_{F_A}$ in terms of the $J$-invariant of $G$ over a base field $F$.

\begin{thm}
\label{main}
Let $J(G_{F_A})=(j'_1, ..., j'_r)$. Then the following holds 
\begin{enumerate}
    \item $j_i=j'_i$ for every component of degree $>1$.
    \item $J^1(G) \cup \{0\} = J^1(G_{F_A}) \cup \{j_{G,A}\}$ as multisets.
   \item In particular, if $J(G) \neq J(G_{F_A}) $, then $j_k \neq j'_k=0$ for some component $j_k$ of degree $1$.
\end{enumerate}

\end{thm}


Let $X$ be a generically split $G$-variety and let $A$ be an algebra which splits over a function field of $X$ (in particular, this is the case for a Tits algebra of $G$). Before proving the main theorem we need to investigate the relation between $F$-rational cycles on the variety $\overline{X\times \SB(A)}$ and $F_A$-rational cycles on the variety $\overline{X}$.
Note that $\overline{\SB(A)} \simeq \mathbb{P}^{\, \deg (A) -1} $
and denote by $h$ the hyperplane class in $\Ch^1(\overline{\SB(A)})$.

\begin{lemma}
\label{lemma1}
Let $y= a_k \times h^k + \Sum_{i>k} a_i\times h^i$  be a homogeneous element in $\Ch(\overline{X\times \SB(A)})$, where $a_i \in \Ch(\overline{X})$. If $y$ is rational over $F$, then 
\begin{enumerate}
\item $1 \times h^k$ is rational over $F$,
\item $a_k$ is rational over $F_A$.
\end{enumerate}
\end{lemma}
\begin{proof}
We first prove the second statement. We have
$$(pr_1)_*(y \cdot (1\times h^i)) = (pr_1)_*(a_k\times [pt] + \dots) = a_k \, ,$$
where $i= \deg (A) -1- k$, $[pt]$ is the class of a rational point in $\Ch(\overline{\SB(A)})$ and $pr_1: \overline{X\times \SB(A)} \rightarrow \overline{X}$ is the projection on the first factor. Since $y$ and $1\times h^i$ are both rational over $F_A$, it follows that $a_k$ is also rational over $F_A$. \\
We now prove the first statement. Consider the surjective pullback
$$f: \Ch(X\times(X\times \SB(A))) \twoheadrightarrow \Ch(X_{F(X\times \SB(A))}) \,.$$
\noindent Since $X$ is split over $F(X \times \SB(A))$, there exists a cycle $a_k^*$ in $\Ch(X_{F(X\times \SB(A))})$, such that $\deg(a_k\cdot a_k^*) =1 $. Let $\alpha$ be the image in   $\Ch(\overline{X\times(X\times \SB(A))})$ of some lifting of $a_k^*$ via $f$. Let $\beta = (pr_2)^*(y)$, where $pr_2: \overline{X\times X \times \SB(A)} \rightarrow  \overline{X\times \SB(A)}$ is the projection on the second factor. We have
$$ \beta = a_k \times 1 \times  h^k + \Sum_{i>k} a_i \times 1 \times h^i \, $$
\noindent and
$$\alpha = a_k^* \times (1 \times 1) + \Sum_{j \in I} b_j \times (c_j) \, ,$$
where $\codim c_j> 0$ and $\codim b_j < \codim a_k^*$ for all $j \in I$. Then
$$ \alpha \cdot \beta = [pt]\times 1 \times h^k + \Sum_{j \in I} b'_j \times (c'_j)\, ,$$
\noindent where $b'_j \in \Ch(\overline{X})$ and $\dim b'_j >0$ for all $j \in I$. Hence, $(pr_{2,3})_*( \alpha \cdot \beta) = 1 \times  h^k$, where $pr_{2,3}: \overline{X\times(X\times \SB(A))} \rightarrow \overline{X\times \SB(A)}$ is the projection on the product of second and third factors. Since both cycles $\alpha$ and $\beta$ are $F$-rational, the cycle $1 \times  h^k$ is also $F$-rational.
\end{proof}

Let $\mathcal{A}=\{a_1, ..., a_s\}$ be a basis of $\F_p$-vector subspace of $F_A$-rational cycles in $\Ch(\overline{X})$. For every $i=1, ..., s$ we fix an $F$-rational lifting $y_i$ of $a_i$ via the surjective pullback.
$$ \Ch(\overline{X\times \SB(A)}) \twoheadrightarrow \Ch(\overline{X_{F_A}}) \,.$$
\noindent Consider the set
\begin{equation}
\label{1xh}
J=\{0\leq j < \deg A \mid 1\times h^j \text{ is }F\text{-rational in }\Ch(\overline{X\times \SB(A)})    \} \, .
\end{equation}

Recall that we denote by $\overline{\Ch}(X\times \SB(A))$ the subring of $F$-rational cycles in $\Ch(\overline{X\times \SB(A)}) $.
We can now decribe a basis of the $\F_p$-vector space $\overline{\Ch}(X\times \SB(A)) $  in terms of $\mathcal{A}$ and $J$. 
\begin{prop}
\label{rat-cycles}
The set $\mathcal{B}=\{ y_i \cdot (1\times h^j) \mid i \in [1,s], j \in J\}$ form a basis of the $\F_p$-vector space  $\overline{\Ch}(X\times \SB(A)) $.
\end{prop}
\begin{proof}
Since $a_i$, $i \in [1,s]$ are linearly independent, the same holds for the elements from $\mathcal{B}$. 
Denote by $\langle \mathcal{B} \rangle$ the $\F_p$-vector subspace in ${\Ch}(\overline{X\times \SB(A)})$ generated by the elements from $\mathcal{B}$. Our goal is to show that $\langle \mathcal{B} \rangle = \overline{\Ch}({X\times \SB(A)})$


Clearly $\langle \mathcal{B} \rangle \subset  \overline{\Ch}({X\times \SB(A)})$. Assume that the subspaces are not equal and among the homogeneous elements in
$\overline{\Ch}({X\times \SB(A)}) \setminus \langle \mathcal{B} \rangle $ choose an element 
$$ y= a \times h^k + \Sum_{i>k} a_i\times h^i, b_i \in \CH(\overline{X}) \, , a \neq 0 \, , $$
\noindent with the maximal $k \geq 0$. 
By Lemma \ref{lemma1} the cycle $a$ is rational over $F_A$ and $k \in J$. Hence, we can write 
$a= \lambda_1 a_1 + \cdots + \lambda_sa_s$ for some $\lambda_i \in \F_p$. Consider the cycle 
$$y'=y - (\lambda_1 y_1 + \cdots + \lambda_sy_s)\cdot(1\times h^k) \, .$$
By our assumption on $y$ we see that $y'=0$ and we get a contradiction. It follows that $\overline{\Ch}({X\times \SB(A)}) =  \langle \mathcal{B} \rangle$ and therefore, $\mathcal{B}$ is indeed a basis of $\overline{\Ch}({X\times \SB(A)})$.
\end{proof}

We are now ready to prove the main theorem.

\begin{proof}[Proof of Theorem \ref{main}]
Recall that we work with Chow groups and motives modulo $p$. 
Without loss of generality we can pass to a field  extension of degree coprime to $p$ and assume
that the index of $A$ is a power of $p$.
Since in this case the algebra $A$ and its $p$-primary component have same splitting fields, we can also assume that the degree $n$ of $A$ is a power of $p$.

Let $X$ be a generically split $G$-variety. 
Then the variety $X\times \SB(A) $ is also generically split for the group $G' =G\times \PGL_1(A)$. Moreover, by Proposition \ref{prop-upper} and our assumption on algebra $A$, we have $\mathcal{R}_p(G) \simeq \mathcal{R}_p(G')$. Applying Proposition \ref{Poincare-rat} for the generically split varieties $X\times \SB(A)$ and $X_{F_A}$ we obtain

$$P(X\times \SB(A), t)= P(\mathcal{R}_p(G), t) \cdot P(\overline{\Ch}(X\times \SB(A)), t)$$
\noindent and
$$P(X_{F_A}, t) = P(\mathcal{R}_p(G_{F_A}), t) \cdot P(\overline{\Ch}(X_{F_A}), t) \, .$$

Let us compare now the left and the right sides of these polynomial equalities. We have
$P(X\times \SB(A), t) = P(X,t) \cdot P(\SB(A), t) =P(X,t) \cdot P(\mathbb{P} ^{n-1}, t) =P(X,t)\frac{t^n-1}{t-1}$. Note also that $P(X,t) =P(X_{F_A},t)$. The first factor on the right hand sides of the equalities can be expressed in terms of the corresponding $J$-invariant  using formula (\ref{formula}). Note that the last factor in the second equality divides the last factor in the first equality by Proposition \ref{rat-cycles} and the quotient is given by the polynomial
$$Q(t)=\Sum_{i \in J} t^i \, ,$$
where the set $J$ is defined in (\ref{1xh}).

Now dividing the first polynomial equality by the second equality we get
\begin{equation}
\label{poly}
\frac{t^n-1}{t-1}= \Prod_{i=1}^{r} \frac{t^{d_ip^{j_i}}-1}{t^{d_ip^{j'_i}}-1} \cdot Q(t) \, .
\end{equation}

The first statement of the theorem follows directly from the above polynomial equality.
Indeed, if $j_i>j'_i$ for some component $i \in \{1,...,r\}$ with $d_i>1$, then the primitive $d_ip^{j_i}$-root of unity $\zeta \in \C$ is a complex root of the right-hand side polynomial from  equality ($\ref{poly}$). However $\zeta$ is not a complex root of the left-hand side polynomial in ($\ref{poly}$), since $n$ is a power of $p$, $d_i>1$ and $p$ does not divide $d_i$.

Before proving statements (2) and (3) of the theorem we will first find the explicit form of the polynomial $Q(t)=\Sum_{i \in J} t^i$, which is defined by the set $J$. It follows from polynomial equality (\ref{poly}), that $Q(t)= t^{\deg Q}Q(1/t)$. Hence, the set $J$ is symmetric with respect to its midpoint. 

Another property 
$$x\in J,\, y\in J \quad \Longrightarrow\quad  x+y\in J$$

\noindent follows from the definition of the set $J$, since the product of two $F$-rational cycles is $F$-rational.

Let $m = \min J\!\setminus \!\{0\}$ (we set $m=n$, if $J=\{0\}$). Using the mentioned above two properties of the set $J$ it is easy to check that
$$Q(t)= 1+t^m+t^{2m}+ ... +t^{(k-1)m}=\frac{t^{km}-1}{t^{m}-1} $$
for some integer $k \geq 1$, such that  $(k-1)m<n \leq km$.

Finally, it follows from (\ref{poly}), that $km$ is a power of $p$ and, thus, is equal to $n$. Moreover, $m$ is also a power of $p$ and we write $m= p^j$. Therefore, the polynomial $Q(t)$ has the following form $\frac{t^{n}-1}{t^{p^j}-1}$.

Using the description of $Q(t)$ and the statement (1) of the theorem and  formula (\ref{poly}) we get
$$(t^{p^j}-1)\Prod_{i=1}^{l}{(t^{p^{j'_i}}-1)} = (t-1) \Prod_{i=1}^{l}(t^{p^{j_i}}-1) \, .$$
\noindent It follows from the above polynomial equality that $J^1(G_{F_A}) \cup \{j\} = J^1(G) \cup \{0\}$ as multisets. Moreover, by the definition of $j$ we have $j=j_{G,A}$ and we obtain the statement (2) of the theorem.

In the case $J(G_{F_A}) \neq J(G)$, by statement (2), we have $j \neq 0$. Hence,  $j_k \neq j'_k=0$ for some $k \in \{1, ...,r\}$, which proves the statement (3) of the theorem.
\end{proof}

With the same notations as in the above theorem, we get the following corollary.
\begin{cor}
Assume that there exists a component $j_i$ of degree $1$ of $J(G)$, such that all other components of degree one are zero (in particular, this is the case when there is only one component of degree $1$). Assume also that $j_i$ becomes zero over $F_A$. Then $j_i = j_{G,A}$.
\end{cor}

\section{$J$-invariant of algebras with involutions}

Let $(A, \sigma)$ be a degree $2n$ central simple algebra over $F$, endowed with an involution of orthogonal type and trivial discriminant. Recall, that the Clifford algebra of $(A, \sigma)$ splits as a direct product $C(A,\sigma)=C_+ \times C_-$ of two central simple algebras over $F$.

Let $G= \PGO^+(A,\sigma)$ be the connected component of the automorphism group of $(A, \sigma)$.
Since $(A,\sigma)$ has trivial discriminant, the group $G$ is an inner twisted form of $G_0= \PGO^+_{2n} $. Both groups $G$ and $G_0$ are adjoint of type $\mathrm{D}_{n}$. Let $\{\omega_1,...,\omega_n\}$ be the respective set of fundamental weights. Note that $A$ is a Tits algebra $A_{\omega_1}$ of $G$. We fix fundamental weights $\omega_+$ and $\omega_-$, which are are a permutation of $\omega_{n-1}$ and $\omega_n$, in such a way that the Tits algebras $A_{\omega_+}$ and $A_{\omega_-}$ are respectively the components $C_+$ and $C_-$ of the Clifford algebra  $C(A, \sigma)$.



Note that in this section we assume $p=2$, which is the only torsion prime of the group $G$. 

Let $X$ be the variety of Borel subgroups in $G$. Recall, that the Picard group $\Pic(\overline{X})$ can be identified with a free $\Z$-module generated by $\omega_i$, $i=1, ...,n$. We denote by $w_i$ the images of $\omega_i$ in $\Ch^1(\overline{X}) = \CH^1(\overline{X}) \otimes \F_2$. 

In \cite{QSZ} Qu\'eguiner-Mathieu, Semenov and Zainoulline introduced the notion of the $J$-invariant of algebras with orthogonal involutions. The $J$-invariant of $(A,\sigma)$ is denoted by $J(A, \sigma)$. By definition, $J(A, \sigma)$ is the $J$-invariant of the respective group $G= \PGO^+(A,\sigma)$, where in the definition of $J(G)$ we take a cocycle whose class corresponds to $(A, \sigma)$ and a designation of the components $C_+$ and $C_-$ (note that the choice of the designation does not affect the value of $J(G)$ see \cite[\S 3]{QSZ}).


The $J$-invariant $J(A,\sigma)$ is an $r$-tuple $(j_1, ... ,j_r)$, where $r=m+1$ if $n=2m$ or $n=2m+1$ (see table \cite[\S 4.13]{PSZ}). Note that the first two components $j_1$ and $j_2$ are of degree $1$ and $d_i=2i-3$ for $i\geq 2$.  For every component $j_i$ we also have an explicit upper bound $k_i$ (see table \cite[\S 4.13]{PSZ}). In particular, $j_1 \leq k_1 = v_2(n)$, where $v_2(-)$ denotes the $2$-adic valuation. According to \cite[\S 3]{QSZ}  one can take  $e_1 = \pi(w_1)$ and $e_2 = \pi(w_+)$ for the generators in $\Ch(G_0)$ corresponding respectively to the components $j_1$ and $j_2$ (see section \ref{J-invariant}).


The goal of this section is to prove Conjecture \ref{conj} and to explicitly compute $j_1$.

By \cite[Corollary 5.2]{QSZ} the first component $j_1$ of $J(A, \sigma)$ is zero if the algebra $A$ is split. We denote by  $F_A$  the function field of the Severi-Brauer variety of $A$, which is a generic splitting field of $A$. By Theorem \ref{main} we have $J(A, \sigma)_{F_A}= (0, j'_2, j_3, ..., j_r)$, that is all components of degree $> 1$ does not change over $F_A$. However, to prove Conjecture \ref{conj} we still need to check that $j_2=j'_2$. In order to show this we will use the following proposition.


\begin{prop}
\label{j_1=jGA}
The first component $j_1$ of the $J(A, \sigma)$ is equal to $j_{G,A}$, where $G= \PGO^+(A,\sigma)$.
\end{prop}
\begin{proof} Let $X$ be the variety of Borel subgroups in $G$.
Assume $n$ is odd. By the fundamental relation we have $[A] = 2[C_+]=2[C_-] \in \Br(F)$, where $C_+$ and $C_-$ are two components of the Clifford algebra of $(A,\sigma)$. Using the exact sequence (\ref{Br}) we see that $w_1 \in \overline{\Ch}(X)$. Hence, by definition $j_1 =j_{G,A}=0 $.

Assume now that $n$ is even. By definition, $j_1$ is the smallest non-negative integer, such that there exists an $F$-rational homogeneous  cycle of the form ${w_1}^{2^{j_1}} + \delta \in \Ch(\overline{X})$, where $\delta $ is an element from  the kernel of the map $\pi: \Ch(\overline{X}) \rightarrow \Ch(G_0)$. Since $n$ is even, we have $[A]+[C_+]+[C_-]=0 \in \Br(F)$. Hence, the cycle $w_1 + w \in \Ch^1(\overline{X})$ is $F$-rational, where we set $w= w_+ +w_-$. Therefore, in the definition of $j_1$ we can replace $w_1$ by $w$.

Recall that the kernel of $\pi$ is an ideal in $\Ch(\overline{X}) \simeq \Ch(G_0/B)$ generated by the rational elements of positive codimension for the twisted form $_{\xi}(G_0/B)$ given by a generic cocycle $\xi$ (see Remark \ref{generic}). Our goal is to show that we can assume $\delta = 0$ in the definition of $j_1$.

Let $Y$ be a variety of isotropic right ideal in $(A, \sigma)$ of reduced dimension $n-1$ (it corresponds to the parabolic subgroup of type $\{1,2,\ldots,n-2\}$).  We have the natural projection $f:X \rightarrow Y$. Since $Y$ is generically split, the projection $f$ is a cellular fibration.  Therefore, by \cite[Theorem 3.7]{PSZ}  the motive $M(Y)$ is a direct  summand of $M(X)$. Moreover, the Chow group of this motivic summand as a subgroup of $\Ch(X)$ can be identified with $\Ch(Y)$ via the pull-back $f^*$, so we can assume $\Ch(Y) \subset \Ch(X)$. Since the cycle $w$ (and also $w^{2^{j_1}}$) is defined in $\Ch(\overline{Y})$, the restriction of the $F$-rational cycle $w^{2^{j_1}} + \delta \in \Ch(\overline{X})$ on the subgroup $\Ch(\overline{Y})$ is also $F$-rational and has the similar form $w^{2^{j_1}} + \tilde{\delta}$, where $\tilde{\delta} \in \Ch(\overline{Y}) \cap \Ker \pi$. Using Proposition \ref{Poincare-rat} and the explicit formula (\ref{formula}) for the Poincar\'e polynomial  $P(\mathcal{R}_2(_\xi G_0), t)$, where $\xi$ is a generic cocycle,  we get that in the generic case a rational cycle in $\Ch(\overline{Y})$ has codimension either zero or at least $2^{k_1} \geq 2^{j_1}$. It follows that any cycle in $\Ch(\overline{Y}) \cap \Ker \pi$ has codimension at least $2^{k_1}$ and any cycle of codimension $2^{k_1}$ is $F$-rational. Hence, $\tilde{\delta}$ is $F$-rational and $w^{2^{j_1}} \in \overline{\Ch}(X)$.
Since $\alpha_X(\omega_+ + \omega_-)=[A]$, by definition of $j_{G,A}$ we have $j_{G,A}=j_1$.
\end{proof}

\begin{cor}
\label{corconj}
Conjecture \ref{conj} holds.
\end{cor}
\begin{proof}
The Conjecture follows from Theorem \ref{main} and the above proposition. 
\end{proof}

Let $J(A,\sigma) =(j_1 ,j_2, \dots ,j_r)$. Then by the above corollary we have $J(A, \sigma)_{F_A}= (0,j_2, \dots, j_r)$. Note that the group $\PGO^+(A,\sigma)$ over the field $F_A$ becomes isomorphic to $\PGO^+(q_{\sigma})$, where $q_{\sigma}$ is the respective quadratic form adjoint to the split algebra with involution $(A, \sigma)_{F_A}$. Hence, we can reduce the computation of the components $j_2, \dots , j_r$ of $J(A,\sigma)$ to the case of quadratic forms. However, over the field $F_A$ we lose information about the first component $j_1$. Our next goal is to find an explicit formula for $j_1$.

We know that $j_1 \leq k_1 = v_2(n)$. Thus, $j_1 =0$ if $n$ is odd. Starting from now in this section we assume that $n$ is a positive even integer. 

Note that in this case all three algebras $A$, $C_+$ and $C_-$ have exponent $2$. Therefore, the indices of these algebras are powers of $2$. We denote by $i_A$ (respectively by $i_+, i_-$) the $2$-adic valuation of the index of $A$ (respectively of $C_+$, $C_-$).




We start the computation of $j_1$ by collecting upper bounds for $j_1$. Clearly we have 
\begin{equation}
\label{upperbound1}
j_1 \leq k_1 \, . 
\end{equation}
By \cite[Corollary 5.2]{QSZ} we have another upper bound for $j_1$
\begin{equation}
\label{upperbound2}
j_1 \leq i_A \, .
\end{equation} 
Finally, the lemma below shows that 
\begin{equation}
\label{upperbound3}
j_1 \leq \max\{i_+, i_-\} \, .
\end{equation}

\begin{lemma}
\label{i+i-}
Inequality (\ref{upperbound3}) holds.
\end{lemma}
\begin{proof}
Let $X$ be a variety of Borel subgroups in $\PGO^+(A,\sigma)$. Let $i=\max\{i_+, i_-\} $. Since $[A]+[C_+]+[C_-]=0 \in \Br(F)$, the cycle $w_1 + w_+ +w_- \in \Ch^1(\overline{X})$ is $F$-rational. We have 
$$(w_1 + w_+ +w_-)^{2^i}= w_1^{2^i} + w_+^{2^i} +w_-^{2^i} \, .$$
By \cite[Proposition 4.2]{PS1} the cycles $w_+^{2^i}$ and $w_-^{2^i}$ are $F$-rational and, hence, $w_1^{2^i}$ is also $F$-rational. It follows from the very definition of $j_1$ that $j_1 \leq i$.
\end{proof}

Combining upper bounds  (\ref{upperbound1}), (\ref{upperbound2}) and (\ref{upperbound3}) we get
\begin{equation}
\label{upperbound}
    j_1\leq \min\{k_1, i_A, \max\{i_+, i_-\}\} \, .
\end{equation}
It appears that the above upper bound is the exact value of $j_1$.

\begin{thm}
\label{j1}
The following formula holds for the first component $j_1$ of $J(A, \sigma)$
\begin{equation}
\label{j1_formula}
 j_1= \min\{k_1, i_A, \max\{i_+, i_-\}\} \, .
\end{equation}

\end{thm}

\begin{proof}
Since inequality (\ref{upperbound}) holds, it remains to show that
$$j_1 \geq \min\{k_1, i_A, \max\{i_+, i_-\}\} \, .$$

Let $X$ be the variety of totally isotropic ideals in $(A, \sigma)$ of reduced dimension $n$ (the variety $X$ is the twisted form of the variety of maximal totally isotropic subspaces of a quadratic form). Since the discriminant of $\sigma$ is trivial, the variety $X$ has two connected components $X^+$ and $X^-$ corresponding to the components $C_+$ and $C_-$ of the Clifford algebra $C(A, \sigma)$. The varieties $X^+$ and $X^-$ are projective homogeneous under the action of the group $\PGO^+(A,\sigma)$.

Recall that $\overline{\SB(A)} \simeq \mathbb{P}^{2n -1} $
and denote by $h$ the hyperplane class in $\Ch^1(\overline{\SB(A)})$. Note that the variety $X= X^+ \times \SB(A)$ and the cycle $1\times h \in \Ch^1(\overline{X^+ \times \SB(A)})$ satisfy the conditions in Definition \ref{jGA} for the group $G=\PGO^+(A,\sigma)$ and central simple algebra $A$. 

Therefore, by Proposition \ref{j_1=jGA}, $j_1$ is equal to the smallest integer $j$, such that $1 \times h ^{2^j} \in \Ch(\overline{X^+ \times \SB(A)})$ is $F$-rational. If $j$ is such a smallest integer, then $h ^{2^j} \in \Ch(\overline{\SB(A)})$ is $F(X^+)$-rational. It follows that $2^j \geq \ind A_{F(X^+)}$. On the other hand, by Index Reduction Formula \cite[page 594]{MPW} we have $\ind A_{F(X^+)} = \min\{2^{k_1}, 2^{i_A}, 2^{i_-} \}$. Hence, the following inequality holds
\begin{equation}
\label{i-}
j_1\geq \min\{{k_1}, {i_A}, {i_-} \} \, .
\end{equation}
Repeating the same arguments but now for the variety $X^- \times \SB(A)$ we get another inequality 
\begin{equation}
\label{i+}
j_1 \geq \min\{{k_1}, {i_A}, {i_+} \} \, .
\end{equation}

It follows from inequalities (\ref{i-}) and (\ref{i+}) that 
$j_1 \geq \min\{k_1, i_A, \max\{i_+, i_-\}\} $ and we conclude that formula (\ref{j1_formula}) holds.
\end{proof}

\begin{cor}
\label{half-spin}
Assume that $(A,\sigma)$ is half-spin, that is one of the components of the Clifford algebra $C(A,\sigma)$ is split.  Then $j_1= \min\{k_1, i_A \}$.
\end{cor}
\begin{proof}
Since $[A]+[C_+]+[C_-]=0 \in \Br(F)$ and one of the components $C_+$ or $C_-$ is split, we have 
$\{i_+, i_-\} = \{i_A, 0 \} $. Then the formula $j_1 = \min\{k_1, i_A \}$ follows from Theorem \ref{j1}.
\end{proof}

The $J$-invaraint of $(A, \sigma)$ contains two components $j_1$ and $j_2$ of degree one. The formula for $j_1$ is computed in Theorem \ref{j1}.  By \cite[Corollary 5.2]{QSZ} we have  $j_2 \leq \min\{i_+, i_-\} $.
It appears that this upper bound is the exact value of $j_2$ in the following case.

\begin{prop}
\label{j2}
If $\min\{i_+, i_-\} < \min\{k_1, i_A\}$, then $j_2= \min\{i_+, i_-\} $.
\end{prop}
\begin{proof}
Denote by $j_i^+$ (resp. $j_i^-$) the $i$-th component of the $J$-invariant of $(A, \sigma)$ over the function field of $\SB(C_+)$ (resp. of $\SB(C_-)$). Note that over such a function field the Clifford invariant of $(A, \sigma)$ is trivial.
By Corollary \ref{half-spin} and using Index Reduction Formula \cite[page 592]{MPW} we have
$$j_1^+ = \min\{k_1, v_2(\ind A_{F(\SB(C_+))}) \} = \min\{k_1, i_A, i_-\} \, . $$
Similarly we get $j_1^- = \min\{k_1, i_A, i_+\} $. 

Recall that at least one of the numbers $i_+$ or $i_-$ is less than $\min\{k_1, i_A\}$. Hence, $j_1^+=i_-$ or $j_1^-=i_+$. Let $\varepsilon$ be a symbol $+$ or $-$, such that $j_1^{\varepsilon}=\min\{i_+, i_-\}$. Observe that $j_2^{\varepsilon}=0$. Therefore, by Theorem \ref{main} we have $j_1^{\varepsilon} \in \{j_1,j_2\}$. If $i_+ \neq i_-$, then by Theorem \ref{j1} we have
$$ j_1 = \min\{k_1, i_A, \max\{i_+, i_-\}\} > \min\{i_+, i_-\} = j^{\varepsilon}_1 $$
\noindent and, hence,  $j_2 =  j^{\varepsilon}_1$.
Assume now $i_+ = i_-$, then $j_1 = j^{\varepsilon}_1 $.
It follows from Theorem \ref{main} that $j_2 =  j_{G, C_{\varepsilon}}  $. 
Considering the variety $X^{-\varepsilon} \times \SB(C_{\varepsilon})$ and applying Index Reduction Formula in the same way as in the proof of Theorem \ref{j1} one can check that $j_{G, C_{\varepsilon}} \geq i_+ = i_-$. 
Therefore, in this case we also have $j_2 = \min\{i_+, i_-\}$.
\end{proof}

\bigskip
{\sc Acknowledgments.} 
I am grateful to Anne Qu\'eguiner-Mathieu for introducing me Conjecture \ref{conj} and for usefull discussions. I would like to thank Nikita Semenov for numerous usefull discussions on the subject of this paper.




\end{document}